\documentclass[tocmacros]{daj}

\newtheorem{observation}[theorem]{Observation}

\dajAUTHORdetails{%
  title = {Composite Ramsey Theorems via Trees}, 
  author = {Matt Bowen},
  plaintextauthor = {Matt Bowen},
}   

\dajEDITORdetails{%
   year={2026},
   number={4},
   received={28 November 2022},   
   published={27 July 2026},  
   doi={10.19086/da.165253},       
}   

\begin{document}

\begin{frontmatter}[classification=text]


\author[bowen]{Matt Bowen\thanks{Supported by Ben Green's Simons Investigator Grant number 376201.}}

\begin{abstract}
We prove a theorem ensuring that the compositions of certain Ramsey families are still Ramsey.  As an application, we show that in any finite coloring of $\mathbb{N}$ there is an infinite set $A$ and an as large as desired finite set $B$ with $(A+B)\cup (AB)$ monochromatic, addressing a problem of Kra, Moreira, Richter, and Robertson.  In fact, we prove an iterated version of this result, which ensures the existence of monochromatic patterns such as $\{a\circ_1 (b \circ_2 c): \circ_i\in \{+,\cdot\}\}, $ generalizing a Ramsey theorem of Bergelson and Moreira that was previously only known to hold for colorings of $\mathbb{Q}$ rather than colorings of $\mathbb{N}$.  Our main new technique is an extension of the color focusing method that involves trees rather than sequences.
\end{abstract}
\end{frontmatter}


\section{Introduction}

Ramsey theory on $\mathbb{N}$ is centered around characterizing the structures of the following form.

\begin{definition}
\normalfont We say a family $\mathcal{B}\subseteq \mathcal{P}(\mathbb{N})$ is \textbf{Ramsey} if for any finite coloring $c:\mathbb{N}\rightarrow [r]$ there is some $B\in \mathcal{B}$ such that $c$ is constant on $B$. 
\end{definition}

The \textbf{linear Rado families}, i.e. the Ramsey families $\mathcal{B}$ that are generated by the set of solutions to a finite system of linear equations, are completely characterized by Rado's theorem \cite{rado1933studien}.  Moreover, given any such family we can obtain a \textbf{geometric Rado family} by composing a given coloring with the homomorphism $n\mapsto 2^n$. 
For example, Schur's theorem tells us that $\{\{x,y,z\}:x+y=z\}$ is a linear Rado family, which in turn implies that $\{\{x,y,z\}:xy=z\}$ is a geometric Rado family.  We will sometimes simply write, for example, that $\{x,y,x+y\}$ is Ramsey, a slight abuse of notation.

In contrast to the linear and geometric case, characterizations of other Ramsey families are still quite elusive, and determining if even some of the simplest non-linear families are Ramsey has been the subject of much recent study; see, e.g., \cite{alweiss2023monochromatic,alweiss2024monochromatic,chow2021rado,di2018ramsey,di2022monochromatic,moreira2017monochromatic,prendiville2021counting,sahasrabudhe2018exponential,sahasrabudhe2018monochromatic}, for some recent results of this type.  Here we further this line of work by showing that certain non-linear compositions of Ramsey families are still Ramsey.  In the below, by $\binom{\mathbb{N}}{m}$ we mean the family of all $m$-element subsets of $\mathbb{N},$ while $\binom{\mathbb{N}}{\infty}$ is the family of all infinite subsets of $\mathbb{N}.$

\begin{theorem}\label{thm:main}
Let $\mathbb{T}$ be the smallest collection of families of subsets of $\mathbb{N}$ such that:

\begin{enumerate}
    \item If $\mathcal{B}$ is a linear or geometric Rado family, then $\mathcal{B}\in\mathbb{T}$
    \item If $\mathcal{B}\in\mathbb{T},$ then for any finite set $P$ of integral polynomials and $m\in\mathbb{N}$ the family $$\{\{a,ab, a+p(b):a\in A, b\in B, p\in P\}: A\in \binom{\mathbb{N}}{m},B\in\mathcal{B}\}\in \mathbb{T}.$$
    
    \vspace{3mm}
    
\end{enumerate}

 If $\mathcal{B}\in \mathbb{T},$ then for any finite set $P$ of integral polynomials the family 

 $$\{\{a,ab, a+p(b):a\in A, b\in B, p\in P\}: A\in \binom{\mathbb{N}}{\infty},B\in\mathcal{B}\}$$

 is Ramsey.

\end{theorem}

\vspace{2mm}

Note in the above that the set $A$ in the conclusion is infinite, but we must restrict to finite sets $A$ when ensuring inclusion in $\mathbb{T}$.

This theorem is somewhat abstract, so we discuss some concrete applications.

For the first, consider the family $\mathcal{B}=\binom{\mathbb{N}}{m}$ consisting of all $m$ element subsets of $\mathbb{N}.$  This is a linear Rado family by the pigeonhole principle, and so applying Theorem \ref{thm:main} with $\mathcal{B}$ and the identity polynomial $n\mapsto n$ we obtain the following:

\begin{corollary}\label{cor:kmrr}
In any finite coloring of $\mathbb{N}$ there is an infinite set $A$ and a set $B\in\binom{\mathbb{N}}{m}$ such that $(A+B)\cup (AB)$ is monochromatic. 
\end{corollary}

The $|A|=|B|=1$ consequence of this corollary is Moreira's theorem \cite{moreira2017monochromatic}. Recently, in \cite{kra2022infinite} Question 8.4, Kra, Moreira, Richter, and Robertson asked if the above result is true when $B$ is also required to be infinite and stated that even the $|A|=|B|=2$ case seemed out of reach.

We can iterate the deduction of Corollary \ref{cor:kmrr} to generalize it to more variables.  For instance, since the family $\binom{\mathbb{N}}{m}\in \mathbb{T}$, by Theorem \ref{thm:main} (2) we see that for any $m\in \mathbb{N}$ the family $\{B\cup (BC) \cup (B+C): B,C\in\binom{\mathbb{N}}{m}\}\in \mathbb{T}.$ The conclusion of Theorem \ref{thm:main} now ensures that there is an infinite set $A$ and $m$-element sets $B,C$ such that the following set is monochromatic: 

\begin{equation}\label{example}
(A+B+C)\cup (A+(BC))\cup (A(B+C))\cup (ABC). 
\end{equation}

 Doing this repeatedly, we obtain:

\begin{corollary}\label{cor iterated}
In any finite coloring of $\mathbb{N}$ there are sets $A_1,...,A_n\in\binom{\mathbb{N}}{m}$ such that $$\{a_1\circ_1 (a_2\circ_2(...\circ_{n-2} (a_{n-1}\circ_{n-1} a_n)...): a_i\in A_i, \circ_i\in \{+,\cdot\}\}$$ is monochromatic.
\end{corollary}

Previously, the above result for $|A_i|=1$ was known to hold in finite colorings of $\mathbb{Q}$ by a theorem of Bergelson and Moreira \cite{bergelson2017ergodic}.  However, even the $n=3$ and $|A_i|=1$ case was open for colorings of $\mathbb{N}.$

So far we have only considered applications of Theorem \ref{thm:main} that begin with the pigeonhole principle as the base family.  Instead starting with a geometric family, we obtain the following common extension of the linear and geometric van der Waerden theorems by applying Theorem \ref{thm:main} with the geometric Rado family $\mathcal{B}=\{zy^i,y: i\leq k\}$ and the linear polynomials $n\mapsto in.$

\vspace{3mm}

\begin{corollary}
The family $\{x+iy,xy,z\cdot xy^i: i\leq k\}$ is Ramsey for any $k\in\mathbb{N}$.
\end{corollary}

Notice that this tells us that any finite coloring of $\mathbb{N}$ contains arbitrarily long arithmetic and geometric progressions of the same color and step size, and moreover with starting points that are a simple multiplicative shift of each other.  One hope might be to prove the following refinement of this result.

\vspace{3mm}

\begin{question}\label{double vdw}
Is the family $\{x+iy,xy^i: i\leq k\}$ Ramsey for any $k\in\mathbb{N}$? 
\end{question}

Our proof of Theorem \ref{thm:main} builds on the ideas used in Moreira's proof that the family $\{xy,x+y\}$ is Ramsey \cite{moreira2017monochromatic}.  We use the fact that there are many potential choices for $x$ and $y$ in each step of the proof to build a tree of possible choices for these values.  From here we deduce new Ramsey theorems by running color focusing arguments on this tree; see especially Figures \ref{fig:focusing1} and \ref{fig:focusing2} for illustrations of this idea.



\section{Notation and technical background}

In this section we collect the notation and technical facts that will be used throughout the paper.  The reader should note that the proofs of the special cases of our main result presented in Section \ref{concrete}, which contain almost all of the combinatorial ideas needed for the proof of our general theorem, mostly avoid these notations and rely on more well known facts.  Namely, the only thing from Subsection \ref{subsection poly} needed in Section \ref{concrete} is the more usual $\mathbb{P}_0=\mathbb{N}$ case of the polynomial van der Waerden theorem, and Subsection \ref{subs family} can be entirely skipped.

The main technical difficulty with the proof of Theorem \ref{thm:main} is that each step of our induction will require us to prove Ramsey theoretic results on the space of polynomials with an additional variable (see the discussion after the proof of Theorem \ref{iterated 2}), which leads to notational difficulties if not combinatorial ones.  

Throughout, we will often abuse notation when using algebraic operations with sets.  For example, if $A$ and $B$ are sets of numbers, by $\frac{A}{B}$ we mean the set $\{\frac{a}{b}:a\in A, b\in B\}.$  Similarly, if $P$ and $Q$ are sets of polynomials, by $P(d)$ we mean the set $\{p(d):p\in P\}$ and by $Q(P(d))$ we mean the set $\{q(p(d)): q\in Q,p\in P\}.$

\subsection{Trees}

In this paper, by a \textbf{tree} we mean a collection of finite words (including the empty word) where the $n$th letter comes from alphabet $A_n.$  More precisely, given finite sets $A_0,...,A_R,$ we'll consider trees of the form $T=\bigcup_{0\leq i\leq R+1}\prod_{j<i} A_j. $  In particular, $T$ is rooted at the empty word, which has $|A_0|$ children each labeled by an element of $A_0$, and so on.

Given a tree $T$ as above, we define the $i$th level of the tree as $T_i=\prod_{j<i} A_j.$  Further, we will use $\smallfrown$ to denote concatenation, i.e., given $t\in T_i$ and $p\in A_i,$ $t\smallfrown p\in T_{i+1}.$

\subsection{Polynomial spaces, notions of size, and van der Waerden's theorem}\label{subsection poly}

As mentioned above, the proof of Theorem \ref{thm:main} will require us to consider spaces of many variabled polynomials with rational coefficients.

\begin{definition}
Let $\mathbb{P}_0=\mathbb{N},$ and having defined $\mathbb{P}_n,$ let $\mathbb{P}_{n+1}$ be the set of all formal objects of the form $\frac{a_1}{b_1}x_{n+1}+\frac{a_2}{b_2}x_{n+1}^2+...+\frac{a_k}{b_k}x_{n+1}^k,$ where $x_{n+1}$ is a formal variable different from the ones used to define $\mathbb{P}_n,$ $k\in\mathbb{N}$, and $a_j,b_j\in \bigcup_{0\leq i\leq n} \mathbb{P}_i,$ and $b_j\neq 0.$ 
\end{definition}

In particular, $\mathbb{P}_1=x\mathbb{Q}_+[x],$ the space of all formal polynomials of degree at least one with positive rational coefficients together with the $0$-polynomial $n\mapsto 0$. 

Given $p\in \mathbb{P}_{n}$ and $d\in\mathbb{P}_m,$ by $p(d)$ we mean the formal object obtained by replacing every instance of $x_{n}$ in $p$ with $d.$  We will only do this in instances where $m=n,$ in which case this is the composition of polynomials, or when $m=n-1,$ where we can think of this as evaluation.  Note that in the latter case in general this might not be an element of $\mathbb{P}_{n-1},$ but in all of our uses of this notation $d$ will have been chosen such that it is.

We will need a piecewise syndetic version of the polynomial van der Waerden theorem for our proofs. To formally state this, we need the following definitions.

\begin{definition}
\hspace{3mm}
\begin{itemize}
    \item A set $S\subseteq \mathbb{P}_n$ is \textbf{syndetic} if there is a finite set $F\subset\mathbb{P}_n$ such that $\mathbb{P}_n\subseteq S-F.$
    
    \item A set $T\subseteq\mathbb{P}_n$ is \textbf{thick} if for any finite set $F,$ there is a $t\in T$ such that $t+F\subset T.$
    
    \item  A set $A\subseteq\mathbb{P}_n$ is \textbf{piecewise syndetic} if there is a thick set $T$ and a syndetic set $S$ with $S\cap T\subseteq A.$
\end{itemize}
\end{definition}  

Unlike for $\mathbb{N},$ $x_{n}\mathbb{P}_n$ is not piecewise syndetic for $n>0.$  This was mentioned by Moreira in \cite{moreira2017monochromatic} Section 7 as a difficulty with extending his proof to rings that are not large ideal domains (such as $\mathbb{P}_n$ for $n>0)$.  However, this fact will not be problematic for us, since we will frontload our use of lemmas about piecewise syndetic sets to before we do any multiplications in all of our proofs.

We now state the two facts about piecewise syndetic sets that we will need.  The first is well known; see, e.g. \cite{hindman2011algebra}, Theorem 4.4.

\begin{lemma}\label{lemma tech}
If $A\subseteq \mathbb{P}_n$ is piecewise syndetic and $A=A_0\cup...\cup A_R,$ then some $A_i$ is piecewise syndetic.
\end{lemma}

Finally, we will use the following variant of the Polynomial van der Waerden theorem.

\begin{theorem}[Polynomial van der Waerden, essentially \cite{bergelson1996polynomial}]\label{pvd}
Let $P\subset \mathbb{P}_{n+1}$ be a finite set of polynomials and $A\subseteq\mathbb{P}_n$ be piecewise syndetic.  Then there is some $d\in\mathbb{P}_n$ such that $$A\cap \bigcap_{p\in P}(A-p(d))$$ is piecewise syndetic.

\end{theorem}

This follows from any variant of the polynomial van der Waerden theorem for commutative semigroups, such as \cite{bergelson2016interplay} Theorem 7.8.  Note that the presence of rational rather than the integral coefficients more typically discussed in this context does not lead to any difficulties.  For example, the existence of monochromatic progressions of the form $a,a+d, a+\frac{d}{2}, a+\frac{d}{3}$ follows from the existence of monochromatic progressions of the form $a, a+2d', a+3d', a+6d'$ and setting $d=6d'.$

\subsection{Ramsey families}\label{subs family}

We will use the following generalization of the linear and geometric Rado families discussed in the introduction.  This generality is mostly present just to isolate the properties of the Ramsey families that are needed for the proof; that is, these are the properties that will be used to glue two Ramsey families together in a way that preserves their structure.

\begin{definition}\label{def ramsey}
Let $\mathcal{B}$ be a family of finite subsets of $\mathbb{P}_n.$
\begin{itemize}
    \item $\mathcal{B}$ is \textbf{$\mathbb{P}_n$-Ramsey} if for any finite coloring $c:\mathbb{P}_n\rightarrow [R]$ there is a $B\in\mathcal{B}$ such that $c$ is constant on $B.$ 
    \item $\mathcal{B}$ is a \textbf{linear Ramsey family} if it is Ramsey and if for every $B\in\mathcal{B}$ and $c\in \mathbb{P}_n,$ $cB\in \mathcal{B}.$
    \item $\mathcal{B}$ is a \textbf{geometric Ramsey family} if it is Ramsey and every $B,C\in\mathcal{B}$ can be expressed as $B=(b_0,...,b_m),$ $C=(c_0,...,c_m)$, where also $(b_0c_0,...,b_mc_m)\in \mathcal{B}.$ Here we insist that this ordering is consistent between elements (for example, the like terms in geometric progressions).
\end{itemize}
\end{definition}

Observe that every linear Rado family is a linear Ramsey family in $\mathbb{P}_n,$ which can be seen by considering the polynomials of the form $\{ix_{n}: i\leq k\}$ for sufficiently large $k,$ and every geometric Rado family is a geometric Ramsey family in $\mathbb{P}_n$ by considering polynomials of the form $\{x_{n}^i: i\leq k\}.$  Throughout this paper we will only be interested in linear and geometric Ramsey families of this form, so we do not specify if they are $\mathbb{P}_n$ Ramsey or $\mathbb{P}_{n+1}$ Ramsey.  This will only be relevant in the proofs of Propositions \ref{geo proof} and \ref{linear proof}. 

The following are representative examples of each type of Ramsey family that is in $\mathbb{T}$ and thus can be composed as in Theorem \ref{thm:main}.  

\begin{example}

\hspace{2mm}
    \begin{enumerate}
        \item The family $\{\{x,y, x+y\}: x,y\in\mathbb{N}\}$ is a linear Ramsey family.  Examples such as this will be analyzed in Proposition \ref{linear proof}.

        \item The family $\mathcal{B}=\{\{x,y,xy^i: i\leq k\}: x,y\in \mathbb{N}\}$ is a geometric family: if $\{x_0,y_0,x_0y_0^i:i\leq k\}$ and $\{x_1,y_1,x_1y_1^i:i\leq k\}$ are in $\mathcal{B}$ then so is $\{x_0x_1, y_0y_1, (x_0x_1)(y_0y_1)^i: i\leq k\}$.  However, it is not a linear Ramsey family.  Examples such as this will be analyzed in Propositions \ref{geom2} and \ref{geo proof}.

        \item The family $\{\{x+y, xy\}: x,y\in\mathbb{N}\}$ is neither linear nor geometric.  However, it is still a member of $\mathbb{T},$ and examples such as this will be analyzed in Propositions \ref{iterated 2} and \ref{prop iterated full}.  
    \end{enumerate}
\end{example}

We will also use the following fact about geometric Ramsey families.

\begin{observation}\label{obs tech}
    Let $\mathcal{B}$ be a geometric Ramsey family and $B_0,...,B_R\in\mathcal{B},$ where we can write $B_i=(b_{0,i},\ldots,b_{m,i})$.  Then defining $b_j=b_{j,0}\cdot b_{j,1}\ldots\cdot b_{j,R}$ for each $j\in \{0,\ldots,m\}$ and $B=\{b_0,\ldots,b_m\},$ we have that $B\in \mathcal{B}$.
\end{observation}

\begin{proof}
    This follows from the definition of being geometric and induction on $R.$ 
    \end{proof}

\section{Two concrete cases}\label{concrete}

In this section we prove two special cases of our main result that contain essentially all of the ideas needed for the general Theorem \ref{thm:main} while requiring much less notation.  Our first example is the following.

\begin{proposition}\label{geom2}
Any $2$-coloring of $\mathbb{N}$ contains a monochromatic set of the form $$\{a,ab,ac,abc, a+b,a+c,a+bc\}.$$
\end{proposition}

This is the two color case of Theorem \ref{thm:main} applied to the geometric Ramsey family $\{b,c,bc\}$ and the polynomials $n\mapsto 0$ and $n\mapsto n$.  The proof here generalizes easily to more colors and arbitrary geometric Ramsey families and polynomials, although we save the details for the next section for the sake of clarity and simplicity of notation.  In particular, the methods here can be easily adapted to prove Corollary \ref{cor:kmrr}.

\begin{proof}
Suppose that $\mathbb{N}=C_0\cup C_1.$  Our goal is to build a tree of infinite sets with edges labeled by polynomials that we will use for a color focusing argument.  

\begin{claim}\label{geo simple claim}
There are finite sets of polynomials $P_0,P_1\subset \mathbb{P}_1,$ elements $d_0,d_1\in \mathbb{N},$ and a tree of infinite monochromatic sets $A_t\subseteq \mathbb{N}$ for $t\in T=\bigcup_{0\leq i\leq 2}\prod_{j<i}P_{j}$ such that
\begin{enumerate}
    \item any $2$-coloring of ${P}_{1}$ contains a monochromatic set of the form $\{b,c,bc\}$ (in this case, $P_1=\{x^i: i\leq s\},$ where $s$ is the $2$-color Schur number, would work).

    \item any $2^{|P_{1}|+1}$-coloring of $P_0$ contains a monochromatic set of the form $\{b,c,bc\}.$
    
    \item  for each $t\in T_i$ there is an $r\in \{0,1\}$ such that $A_t\cup\bigcup_{k\geq i}(A_t+\prod_{i\leq j\leq k}P_j(d_j))\subseteq C_r.$

    \item for $t\in T_i$ and $p_i\in P_i$ we have $p_i(d_i)A_t=A_{t\smallfrown p_i}.$
\end{enumerate}
\end{claim}

\vspace{2mm}

We will call structures similar to the one described above \textbf{color focusing trees}.  Before proving the claim, let us see how to use it to finish the proof of Proposition \ref{geom2} (see, especially, Figure \ref{fig:focusing1}). 

\begin{figure}
    \centering
    \includegraphics[scale=2.5]{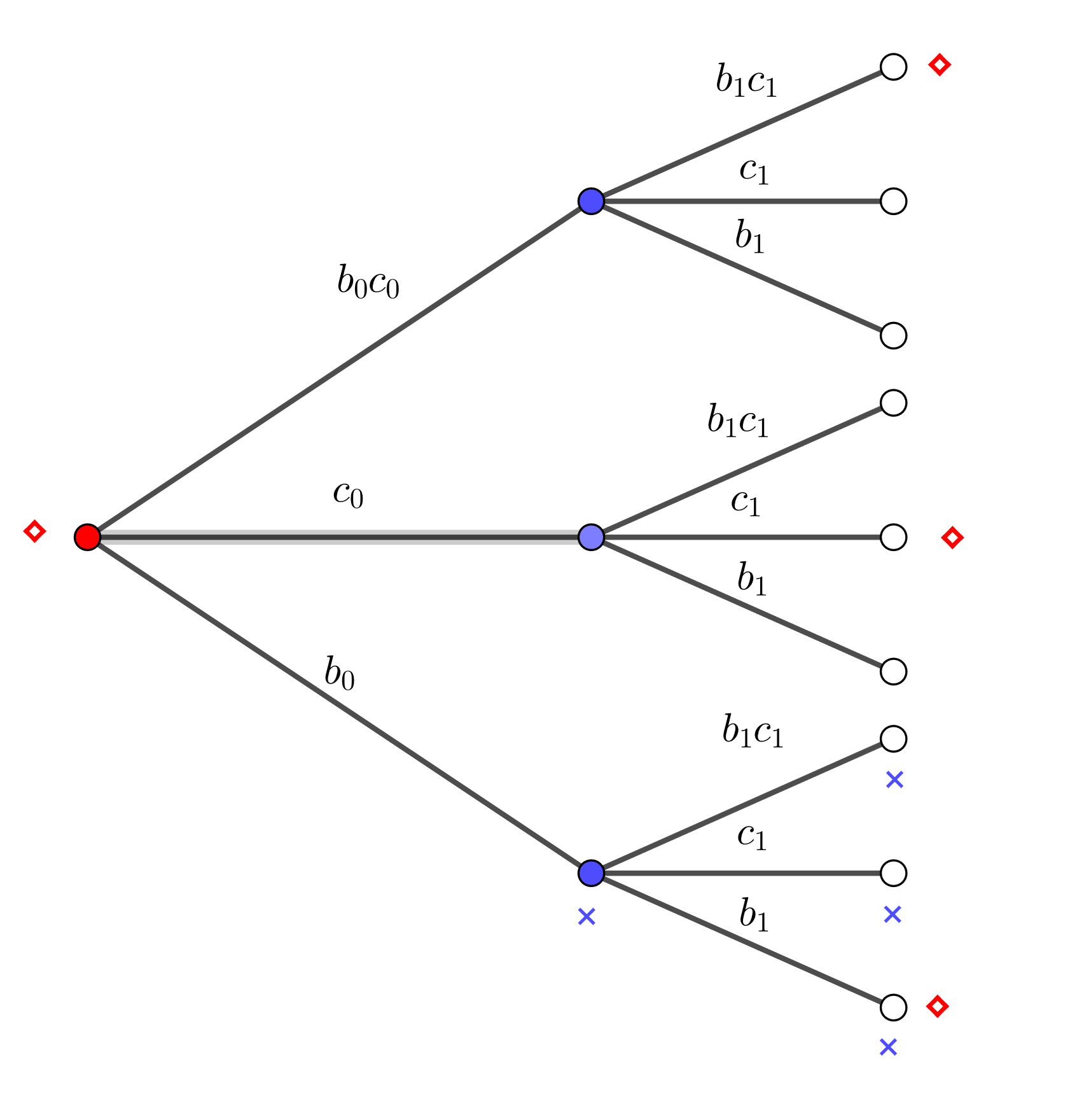}
    \caption{The color focusing step of the proof of Proposition \ref{geom2}.  Each vertex of the tree represents an infinite monochromatic set and multiplying by the indicated polynomials sends elements of a vertex to its child.  The uncolored vertices are all the same color.  If these vertices are blue then we can finish the proof by setting $b={b_1}(d_1)$ and $c={c_1}(d_1)$ and considering the vertices marked with ${\color{blue} \times}.$  Otherwise they are red and we can finish by setting $b={b_0}(d_0)\cdot b_1(d_1)$ and $c={c_0}(d_0)\cdot {c_1}(d_1)$ and considering the vertices marked with ${\color{red} \diamondsuit}.$}
    \label{fig:focusing1}
\end{figure}

Without loss of generality we may assume that $A_\emptyset\subseteq C_0.$  Note also that by Property (3) of Claim \ref{geo simple claim} that each set $A_t$ for $t\in T$ is monochromatic, and hence we can $2$-color the $t\in T$ based on the color of $A_t$.  Now, consider an auxiliary $2^{|P_1|+1}$-coloring of the $p\in P_0$ based on the $(|P_1|+1)$-tuple consisting of the color of $A_p$ together with the colors of its $|P_1|$ descendants in the tree.  By Property (2) of the Claim we find $b_0,c_0\in P_0$ such that the set $P_0'=\{{b_0}, {c_0}, {b_0c_0}\}$ is monochromatic according to this coloring.  By Property (4) of the Claim, this means that for any $a_0\in A_\emptyset$ the set $\{a_0p(d_0): p\in P'_0\}$ is monochromatic.

If this set is monochromatic in $C_0$ then we can finish the proof by letting $a\in A_\emptyset\subseteq C_0,$ $b={b_0(d_0)},$ and $c={c_0(d_0)}.$ This is because $\{ab,ac,abc\}\subset C_0$ by the above paragraph and $\{a+b,a+c,a+bc\}\subset A_\emptyset+P_0(d_0)\subseteq C_0$ by Property (3) of the Claim.

Otherwise, we can assume that $A_{p}\subseteq C_1$ for each $p\in P_0'.$  By our construction, for $p_1\in P_1$ and $p'_0
\in P_0',$ the color of $p_1(d_1)A_{p'_0}$ depends only on $p_1.$  Considering the $2$-coloring of $p\in P_1$ based on the color of $A_{p'\smallfrown p}$ and using Property (1) of the claim, we find $b_1,c_1\in P_1$ such that $P_1'=\{{b_1}, {c_1}, {b_1c_1}\}$ is monochromatic.  This means that the sets $A_{p'_0\smallfrown p'_1}$ for $p'_0\in P'_0$ and $p'_1\in P'_1$ are all monochromatic in the same color.

If $p'_1(d_1)A_{p_0'}\subseteq C_1$ for $p'_1\in P'_1$ and $p'_0\in P'_0,$ then we can finish as above by taking $a\in A_{p_0'},$ $b={b_1(d_1)},$ $c={c_1(d_1)}$, and applying Properties (3) and (4) of the Claim as done above.

Otherwise, $p'_1(d_1)A_{p_0'}\subseteq C_0$ for $p'_1\in P'_1$ and $p'_0\in P'_0.$  In this case we can finish by setting $a\in A_\emptyset,$ $b={b_0}(d_0)\cdot {b_1(d_1)},$ and $c={c_0(d_0)}\cdot {c_1(d_1)},$ and applying Properties (3) and (4) of the Claim as above.

Now, in order to finish the proof of Proposition \ref{geom2} we just need to prove Claim \ref{geo simple claim}.

\begin{proof}[Proof of Claim \ref{geo simple claim}]
The construction of the sets $A_t$ is routine and follows from a couple of applications of the polynomial van der Waerden theorem.  If we were interested in finding a sequence rather than a tree of infinite sets satisfying the conditions of the Claim, then the construction would be essentially the same as that used by Moreira in \cite{moreira2017monochromatic} Corollary 1.5.

First we define $P_1$ and $P_0.$ Let $s_1$ be the $2$-color Schur number and define $P_1=\{x^i: i\leq s_1\}.$  Let $s_0$ be the $2^{s_1+1}$-color Schur number and $P_0=\{x^i: i\leq s_0\}.$  These satisfy properties (1) and (2).

Now, without loss of generality we may assume that $C_0$ is piecewise syndetic by Lemma \ref{lemma tech}.  By Theorem \ref{pvd}, we know that there is a $d_0\in \mathbb{N}$ such that the set $A_0=C_0\cap\bigcap_{p\in P_0} (C_0-p(d_0))$ is piecewise syndetic.  Consider the $2^{|P_0|}$-coloring of $a_0\in A_0$ based on the $|P_0|$-tuple listing the colors of $a_0\cdot p_0(d_0)$ for $p_0\in P_0.$  By Lemma \ref{lemma tech}, one of these color classes is piecewise syndetic.  Let this set be $A_0'.$

Apply Theorem \ref{pvd} to $A'_0$ with the set of polynomials $$P_1'=\{p_0(d_0) p_1,\frac{p_1}{p_0(d_0)}: p_i\in P_i\} $$ to find $d_1\in \mathbb{N}$ such that $A_1=A_0'\cap\bigcap_{p_1\in P_1'} (A_0'-p_1(d_1))$ is piecewise syndetic.  Finally, consider the $2^{|P_0|\cdot|P_1|}$ coloring of the $a_0\in A_1$ based on the tuple listing the color of $a_0\cdot p_0(d_0)\cdot p_1(d_1),$ and let $A_\emptyset$ be the color class that is piecewise syndetic by Lemma \ref{lemma tech}.  Letting $A_t$ be defined to satisfy Property (4) of the claim gives the desired sets.
\end{proof}

\end{proof}

Our next special case is the first iterative case of Theorem \ref{thm:main}, i.e. the $|A|=|B|=|C|=1$ case of Equation \ref{example}.

\begin{proposition}\label{iterated 2}
Any $2$-coloring of $\mathbb{N}$ contains a monochromatic set of the form $$\{a+b+c,a+bc,a(b+c),abc\}.$$
\end{proposition}

As in the proof of Proposition \ref{geom2}, our strategy will be to build a color focusing tree.  We will need the following version of Moreira's theorem \cite{moreira2017monochromatic}, whose proof requires only a slight variation on Moreira's arguments.  

\begin{theorem}\label{mor}
Let $F\subset \mathbb{N}$ be finite.  For any finite coloring of $\mathbb{P}_1$ there are $x,y\in\mathbb{P}_1$ such that $$\{xy, x+\frac{y}{f}:f\in F\}$$ is monochromatic.  
\end{theorem}

We will skip the proof of this theorem for now; the only needed change to Moreira's arguments is that all of the applications of van der Waerden's theorem are done at the start of the proof before any multiplications.  This avoids the problem caused by the set $x\mathbb{Q}[x]$ not being piecewise syndetic in $\mathbb{Q}[x].$   Theorem \ref{mor} will also follow from Proposition \ref{geo proof} below applied to the trivial Ramsey family $\mathcal{B}=\binom{\mathbb{P}_1}{1}.$  Alternatively, Theorem \ref{mor} directly follows from \cite{alweiss2024monochromatic} Theorem 1.2, which appeared after the initial publication of this article. 

\begin{proof}[Proof of Proposition \ref{iterated 2}]
Fix a $2$-coloring $\mathbb{N}=C_0\cup C_1.$ We will use the following color focusing tree.

\begin{claim}\label{tree iterated 2}
There are finite sets $P_0,P_1\subset \mathbb{P}_1$, naturals $d_0,d_1\in \mathbb{N},$ and a tree of infinite monochromatic sets $A_t\subseteq \mathbb{N}$ for $t\in T=\bigcup_{0\leq i\leq 2}\prod_{j<i}P_j$ satisfying:

\begin{enumerate}
    \item any $2$-coloring of $P_0$ contains a monochromatic set of the form $\{b_0c_0, b_0+c_0\}.$
    
    \item any $2$-coloring of $P_1$ contains a monochromatic set of the form $$\{b_1c_1, b_1+c_1, b_1+\frac{c_1}{p_0(d_0)}: p_0\in P_0\}.$$
    
    \item for each $t\in T_i$ there is an $r\in \{0,1\}$ such that $A_t\cup\bigcup_{k\geq i}(A_t+\prod_{i\leq j\leq k}P_j(d_j))\subseteq C_r.$
    
    \item for $t\in T_i$ and $p_i\in P_i$ we have $p_i(d_i)A_t=A_{t\smallfrown p_i}.$
\end{enumerate}
\end{claim}

The proof of Claim \ref{tree iterated 2} is nearly identical to that of Claim \ref{geo simple claim}, with the only real difference being that instead of Schur's theorem we use Theorem \ref{mor} and compactness to ensure that polynomials $P_i\subseteq \mathbb{P}_1$  satisfying Properties (1) and (2) exist.  We omit the details.

We now use the Claim to complete the proof of Proposition \ref{iterated 2}; see especially Figure \ref{fig:focusing2}. 

\begin{figure}
    \centering
    \includegraphics[scale=2.5]{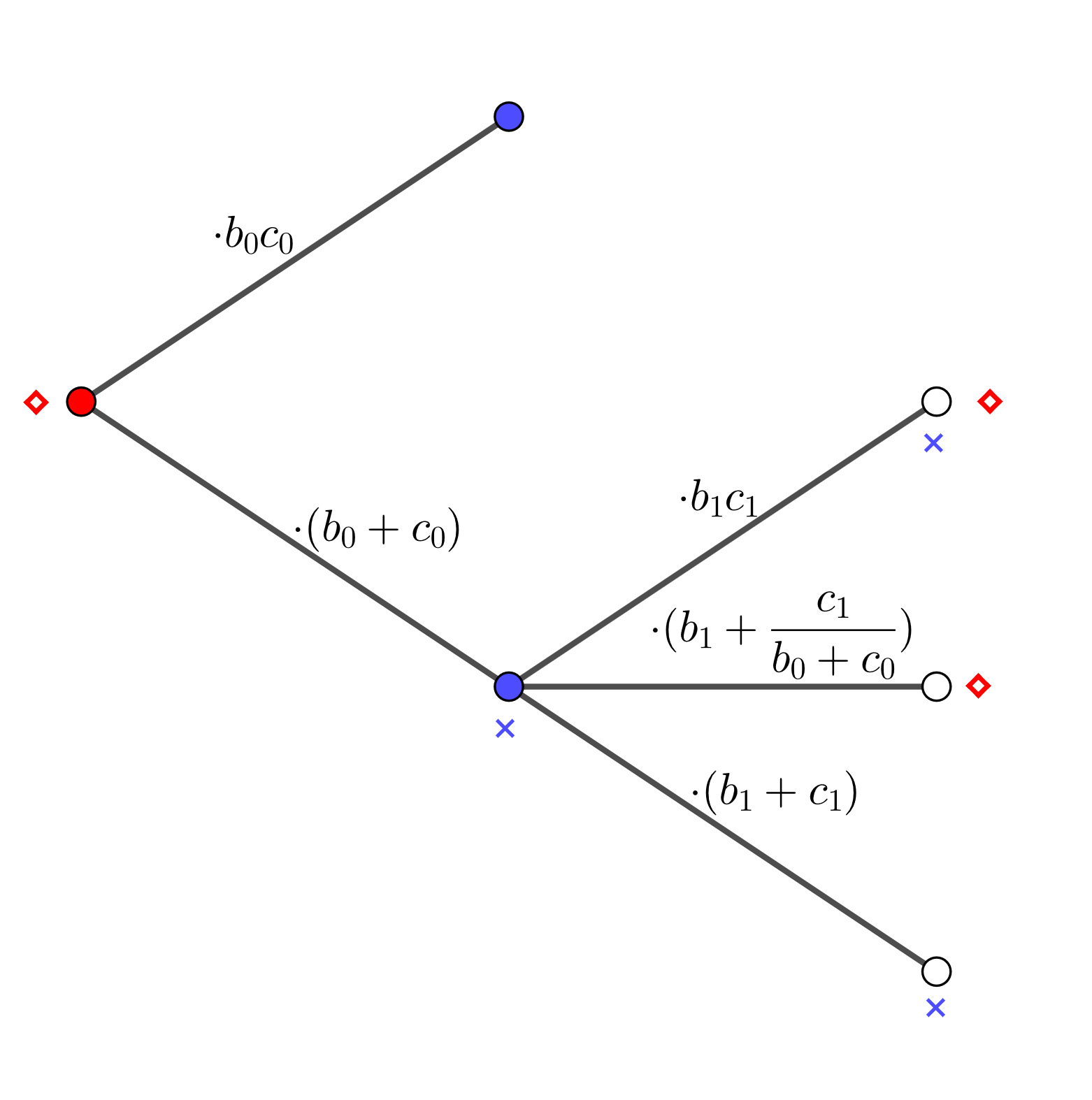}
    \caption{The color focusing tree from the proof of Proposition \ref{iterated 2}.  The uncolored vertices form a monochromatic set.  If they are blue we can finish the proof by setting $b=b_1(d_1)$ and $c=c_1(d_1)$ and considering the vertices marked with ${\color{blue}\times}.$  Otherwise they are red and we can finish by setting $b=(b_0(d_0)+c_0(d_0))\cdot b_1(d_1)$ and $c=c_1(d_1)$ and considering the vertices marked with ${\color{red} \diamondsuit}.$}
    \label{fig:focusing2}
\end{figure}

Without loss of generality we may assume that $A_\emptyset\subseteq C_0.$  $2$-coloring elements $p_0\in P_0$ based on the color of $A_{p_0},$ by Property (1) of Claim \ref{tree iterated 2} we find a monochromatic set $P_0'=\{b_0c_0,b_0+c_0\}\subset P_0$.  

If this set is monochromatic in $C_0$ then define $b=b_0(d_0), c=c_0(d_0),$ and $a\in A_\emptyset.$ We know $abc$ and $a(b+c)\in C_0$ by Property (4) of Claim \ref{tree iterated 2} and the fact that $P_0'$ is monochromatic.  Moreover, $a+b+c, a+bc\in A_\emptyset + P_0(d_0)\subseteq C_0$ by Property (3) of the Claim, so these are as desired.

Otherwise, we know $A_{(b_0+c_0)}\subseteq C_1.$  $2$-coloring elements $p_1\in P_1$ based on the color of $A_{(b_0+c_0)\smallfrown p_1},$ by Property (2) of Claim \ref{tree iterated 2} we find a monochromatic set $P_1'=\{b_1c_1,b_1+c_1, b_1+\frac{c_1}{b_0(d_0)+c_0(d_0)}\}\subset P_1.$  

If this set is monochromatic in $C_1$ then we can finish as above by setting $b=b_1(d_1),$ $c=c_1(d_1),$ and $a\in A_{b_0+c_0}$ and applying properties (3) and (4) of Claim \ref{tree iterated 2} as above.

Otherwise, $P_1'$ is monochromatic in $C_0.$  Now set $b=(b_0(d_0)+c_0(d_0))\cdot b_1(d_1),$ $c=c_1(d_1),$ and $a\in A_\emptyset.$  Then $abc\in C_0$ and $a(b+c)\in C_0$ by Property (4) of Claim \ref{tree iterated 2}.  Finally, observe that $$a+b+c=a+(b_0(d_0)+c_0(d_0))(b_1(d_1)+\frac{c_1(d_1)}{(b_0(d_0)+c_0(d_0))})\in A_\emptyset+P_0(d_0)P_1(d_1)\subseteq C_0$$ by Property (3) of Claim \ref{tree iterated 2}.  Similarly, $$a+bc=a+(b_0(d_0)+c_0(d_0))b_1(d_1)c_1(d_1)\in A_\emptyset+ P_0(d_0)P_1(d_1)\subseteq C_0,$$ completing the proof.   
\end{proof}

Before concluding this section and moving on to the proof of our general theorem, notice how in the proof of Proposition \ref{iterated 2} we needed to use a version of Moreira's theorem \ref{mor} for colorings of $\mathbb{P}_1.$  If we wanted to prove the $n=3$ case of Corollary \ref{cor iterated} we would need to prove a version of Proposition \ref{iterated 2} for colorings of $\mathbb{P}_1$ rather than $\mathbb{N}$ for a similar reason, and this in turn would require us to prove a version of Moreira's theorem for colorings of $\mathbb{P}_2.$  This will be the main source of technical difficulties in the next section, although it will lead to mostly notational rather than combinatorial problems.

\section{The general case}

In this section we prove Theorem \ref{thm:main}.  As mentioned previously, all of the combinatorial ideas needed for the proof are already present in the two concrete cases discussed in Section \ref{concrete}, so we suggest that the reader focus on understanding those proofs.  We only include the technical proofs in this section for the purpose of completeness and verification.

The proof will split into three independent cases depending on the structure of the family $\mathcal{B}$.  Throughout we will need to work in the space of many variabled rational polynomials for the reasons discussed after the proof of Proposition \ref{iterated 2}.  

\begin{proposition}\label{geo proof}
Let $\mathbb{P}_n=C_1\cup...\cup C_R$ be a finite coloring, $P\subset \mathbb{P}_{n+1}$ a finite set of polynomials, and $\mathcal{B}$ a geometric Ramsey family.  There is an infinite $A\subseteq \mathbb{P}_n$ and a $B\in\mathcal{B}$ such that $$\{a,ab,a+p(b): a\in A, b\in B, p\in P\}$$ is monochromatic.  
\end{proposition}

\begin{proof}

As in the special cases proved in Section \ref{concrete}, our goal is to build a color focusing tree.

\begin{claim}\label{geo}
There are finite sets of polynomials $P_0,...,P_R\subset \mathbb{P}_{n+1},$ elements $d_0,...,d_R\in \mathbb{P}_n,$ and a tree of infinite monochromatic sets $A_t$ for $t\in T=\bigcup_{0\leq i\leq R+1}\prod_{j<i}P_{j}$ such that

\begin{enumerate}
    \item any $R$-coloring of ${P}_{R}$ contains a monochromatic element of $\mathcal{B}.$
    \item for $0\leq i<R,$ any $R^{|P_{i+1}|\cdot...\cdot|P_R|+1}$-coloring of $P_i$ contains a monochromatic element of $\mathcal{B}.$
    \item  for each $0\leq i\leq R$ and each $t\in T_i$ there is a $c\in \{1,...,R\}$ such that $$A_t\cup\bigcup_{i\leq k\leq R}(A_t+P(\prod_{i\leq j\leq k}P_j(d_j)))\subseteq C_c.$$
    \item for $t\in T_i$ and $p_i\in P_i$ we have $p_i(d_i)A_t=A_{t\smallfrown p_i}.$
\end{enumerate}
\end{claim}

We can use Claim \ref{geo} to finish the proof in exactly the same manner as we finished the proof of Proposition \ref{geom2} (see especially Figure \ref{fig:focusing1}).  We include the details here more for the purpose of verification and completeness than understanding.

By coloring vertices $t\in T$ based on the color of $A_t$ and the color of all of their descendants in the tree and applying Properties (1) and (2) of Claim \ref{geo}, we find subsets $B_0\subset P_0,$ $\ldots,$ $B_R\subset P_R$ such that $B_k\in\mathcal{B}$ for $k\in \{0,...,R\}$ and such that for any sequence $b_0,b_1,\ldots, b_j$ with $b_k\in B_k$, the color of $A_{b_0\smallfrown...\smallfrown b_j}$ depends only on $j$ (and not on the choices of $b_k\in B_k$).  By the pigeonhole principle we find $i<j$ such that these colors agree on some color $c\in [R]$, i.e. the sets $A_{b_0\smallfrown...\smallfrown b_i}$ and $A_{b_0\smallfrown...\smallfrown b_j}$ are all subsets of $C_c$ for any choices of $b_k\in B_k$.  

Suppose that the sets $B_k\in\mathcal{B}$ for $k\in \{0,...,R\}$ all have cardinality $m+1$ and that $B_k$ can be written as $(b_{0,k},...,b_{m,k}),$ where the indexing is as in the definition of being a geometric Ramsey family, Definition \ref{def ramsey}.  Let $t=b_{0,0}\smallfrown...\smallfrown b_{0,i},$ where $i$ is the smaller index obtained via the pigeonhole principle in the previous paragraph.  To finish the proof we will define $A=A_t\subseteq C_c$ and for $0\leq k\leq m$ we define $b_k=b_{k,i+1}(d_{i+1})\cdot...\cdot b_{k,j}(d_j).$  Then since $\mathcal{B}$ is a geometric Ramsey family, we know $\{b_0,...,b_m\}=B\in \mathcal{B}$ by Observation \ref{obs tech}.  Moreover, by the previous paragraph and Property (4) of Claim \ref{geo} we know $ab_k\in C_c$ for any $b_k\in B$ and $a\in A.$  Finally, by Property (3) we know that $a+P(b_k)\subseteq C_c$ for each $a\in A$ and $b_k\in B,$ finishing the proof. 

\vspace{2mm}

Therefore, in order to complete the proof of Proposition \ref{geo proof} it suffices to prove Claim \ref{geo}.

\begin{proof}[Proof of Claim \ref{geo}]

All of the combinatorial ideas needed for the proof of the Claim were already in the proof of Claim \ref{geo simple claim}; the only new difficulty is the notation.   We will use several applications of the polynomial van der Waerden theorem.

First, we construct the polynomials $P_R,...,P_0\subset \mathbb{P}_{n+1}.$  By the definition of geometric Ramsey families and the compactness principle, we know that for any $k\in \mathbb{N}$ there is a finite set $F\subset \mathbb{P}_{n+1}$ such that any $k$-coloring of $F$ contains a monochromatic element of $\mathcal{B}.$  Let $P_R$ be such a finite set for $k=R,$ and having defined $P_R,P_{R-1},...,P_{i+1}$ let $P_i$ be such a set for $k=R^{|P_{i+1}|\cdot...\cdot|P_R|+1}.$  These satisfy Properties (1) and (2) of the claim.

For the other properties, suppose without loss of generality that $C_1$ is piecewise syndetic, and let $A'_{-1}=C_1.$  We inductively define a sequence of piecewise syndetic sets $A'_{-1}\supseteq A_0\supseteq A_0'\supseteq A_1\supseteq A_1'\supseteq... \supseteq A_{R}',$ elements $d_0,...,d_R\in \mathbb{P}_n,$ and finite sets of polynomials $P_0',...,P_R'\subset \mathbb{P}_{n+1}$ such that:

\begin{enumerate}
    \item[(a)] $$P_i'= \bigcup_{0\leq a\leq b\leq c<i} \frac{P(P_i\cdot\prod_{b\leq j\leq c} P_j(d_j))}{\prod_{ j' <a}P_{j'}(d_{j'})},$$ where empty products have value $1.$
    \item[(b)] $A_i=A_{i-1}'\cap\bigcap_{p'\in P_i'} (A_{i-1}'-p'(d_i)),$ where $d_i$ is chosen such that this set is piecewise syndetic.
    \item[(c)] Coloring $a\in A_i$ based on the tuple listing the colors of $a\cdot p_0(d_0)\cdot...\cdot p_i(d_i)$ for $p_j\in P_j,$ $A_i'$ is chosen to be a piecewise syndetic monochromatic subset of $A_i.$ 
\end{enumerate}

Finding sets satisfying (b) is possible by the polynomial van der Waerden theorem \ref{pvd}, and finding sets satisfying (c) is possible by Lemma \ref{lemma tech}.

Once these have been constructed, set $A_\emptyset=A_{R}'$ and define $A_t$ for $t\in T$ to satisfy (4) of Claim \ref{geo}.  Note that property (2) of the Claim follows from (b) and property (3) from (c).

\end{proof}

\end{proof}

The proofs of the other two cases use color focusing arguments similar to the one used in Proposition \ref{iterated 2} and Figure \ref{fig:focusing2}.

We start with the linear case, which is particularly simple.

\begin{proposition}\label{linear proof}
Let $\mathbb{P}_n=C_1\cup...\cup C_R$ be a finite coloring, $P\subset \mathbb{P}_{n+1}$ a finite set of polynomials, and $\mathcal{B}$ a linear Ramsey family.  There is an infinite $A\subseteq \mathbb{P}_n$ and a $B\in\mathcal{B}$ such that $$\{a,ab,a+p(b): a\in A, b\in B, p\in P\}$$ is monochromatic.  
\end{proposition}

\begin{proof}
As usual, we build a color focusing tree.

\begin{claim}\label{arith claim}
There are finite sets of polynomials $P_0,...,P_R\subset \mathbb{P}_{n+1},$ elements $d_0,...,d_R\in \mathbb{P}_n,$ and a tree of infinite monochromatic sets $A_t\subseteq \mathbb{P}_n$ for $t\in T=\bigcup_{0\leq i\leq R+1}\prod_{j<i}P_{j}$ such that

\begin{enumerate}
    \item for any $i\in \{0,...,R\}$, any $R$-coloring of ${P}_i$ contains a monochromatic element of $\mathcal{B}.$
    \item  for each $t\in T_i$ there is a $c\in \{1,...,R\}$ such that $$A_t\cup\bigcup_{i\leq k\leq R}(A_t+P(\prod_{i\leq j\leq k}P_j(d_j)))\subseteq C_c.$$
    \item for $t\in T_i$ and $p_i\in P_i$ we have $p_i(d_i)A_t=A_{t\smallfrown p_i}.$
\end{enumerate}
\end{claim}

We will omit the proof of Claim \ref{arith claim} as it follows exactly the same steps as the constructions in Claims \ref{geo} and \ref{geo simple claim} above.

To finish the proof, by Property (1) and (3) of Claim \ref{arith claim} we find subsets $B_i\subset P_i$ with $B_i\in \mathcal{B}$ and a sequence $b_0\in B_0,\ldots,b_R\in B_R$ such that for any $j\in \{0,\ldots, R\}$ the color of $A_{b_{0}\smallfrown b_1\smallfrown\ldots\smallfrown b_{j-1}\smallfrown b'_j}$ depends only on $j$ and not on the choice of $b'_j\in B_j$.  By the pigeonhole principle we find $i<j\in \{0,\ldots,R\}$ such that the colors of $A_{b_{0}\smallfrown...\smallfrown b_{i}}$ and $A_{b_{0}\smallfrown ...\smallfrown b_{j-1}\smallfrown b_{j}'}$ agree for any choice of $b_{j}'\in B_j$.

Let $b=\prod_{k=i+1}^{j-1}b_{k}(d_k).$  Then $bB_j\in \mathcal{B}$ since $B_j\in\mathcal{B}$ and $\mathcal{B}$ is linear. We are done by defining $A=A_{b_{0}\smallfrown...\smallfrown b_{i}},$ $B=bB_{j}(d_j),$ and applying Property (2) of Claim \ref{arith claim}. 

\end{proof}

We are now ready to prove the iterated case.

\begin{proposition}\label{prop iterated full}
Let $\mathcal{B}$ be a Ramsey family such that for any $n,m\in\mathbb{N}$ and any finite $P'\subset \mathbb{P}_{n+1}$ the family $$\mathcal{D}=\{\{xb, x+p'(b):b\in B, x\in X, p'\in P'\}: X\in \binom{\mathbb{P}_n}{m},B\in\mathcal{B}\}$$ is $\mathbb{P}_n$-Ramsey.

Then for any $n\in\mathbb{N}$ and any finite $P\subset \mathbb{P}_{n+1}$ the family 

$$\{\{ad, a+p(d):d\in D, a\in A, p\in P\}: A\in \binom{\mathbb{P}_n}{\infty},D\in\mathcal{D}\}$$ is $\mathbb{P}_n$-Ramsey.
\end{proposition}

\begin{proof}
Fix $\mathcal{D},$ $\mathcal{B},$ and $P$ as above and consider a finite coloring $\mathbb{P}_n=C_1\cup...\cup C_R.$  We will assume that $P$ contains the $0$-polynomial $p(d)=0.$

The proof is a (notationally) more technical version of the proof of Proposition \ref{iterated 2} but follows from exactly the same logic.  We will make use of the following color focusing tree.

\begin{claim}\label{iterated full}
There are finite sets of polynomials $P_0,...,P_R\subset \mathbb{P}_{n+1},$ elements $d_0,...,d_R\in \mathbb{P}_n,$ and a tree of infinite monochromatic sets $A_t$ for $t\in T=\bigcup_{0\leq i\leq R+1}\prod_{j<i}P_{j}$ such that

\begin{enumerate}
    \item any $R$-coloring of ${P}_{0}$ contains a monochromatic element of $\mathcal{D}.$
    \item for $0\leq r\leq R$ and any $R$-coloring of $P_r,$ there is a $B_r\in\mathcal{B}$ and an $X_r\in \binom{\mathbb{P}_{n+1}}{m}$ such that 
    
    $$B'_r=\{x_rb_r, x_r+\frac{p(b_r)}{\prod_{i\leq j \leq k<r} P_j(d_j)}: x_r\in X_r, b_r\in B_r, p\in P, 0\leq i\leq k<r\}\subseteq P_r$$
    
    is monochromatic, where if $k=0$ the (empty) product in the denominator is $1$.
    
    \item  for each $t\in T_i$ there is a $c\in \{1,...,R\}$ such that $$A_t\cup\bigcup_{i\leq k\leq R}(A_t+P(\prod_{i\leq j\leq k}P_j(d_j)))\subseteq C_c.$$
    \item for $t\in T_i$ and $p_i\in P_i$ we have $p_i(d_i)A_t=A_{t\smallfrown p_i}.$
\end{enumerate}
\end{claim}

Again, we will omit the proof since it follows from only slight modifications of the proofs of Claims \ref{geo simple claim} and \ref{geo}.

The deduction of Proposition \ref{prop iterated full} from Claim \ref{iterated full} is just a natural generalization of the similar deduction in the proof of Proposition \ref{iterated 2}.     

Consider the sequence $B_0',...,B_R'$ of sets as in Properties (1) and (2) of Claim \ref{iterated full} and an arbitrary choice of elements $x_r\in B_r'$ for each $r\in \{0,\ldots,R\}$.  Notice that by Properties (3) and (4) of Claim \ref{iterated full} we know that the color of $A_{x_0\smallfrown x_1\smallfrown\ldots\smallfrown x_{r-1}\smallfrown b_r'}$ for $b'_r\in B_r'$ depends only on our choice of $r\in \{0,...,R\}$ (and not on the choice of $b'_r\in B_r'$).   

By the pigeonhole principle, we find $i<j$ such that these colors agree, i.e. there is some $c\in \{1,...,R\}$ with $A_{x_0\smallfrown...\smallfrown x_i}\subseteq C_c$ and $A_{x_0\smallfrown...\smallfrown x_{j-1}\smallfrown b_j'}\subseteq C_c$ for each $b_j'\in B_j'.$

To finish the proof, set $A=A_{x_0\smallfrown...\smallfrown x_i}$ and $$C'=x_{i+1}(d_{i+1})\cdot ...\cdot x_{j-1}(d_{j-1})\cdot B_j'(d_j).$$  By the definition of $B_j'$ in property (2) of Claim \ref{iterated full}, $C'$ contains a subset $D\in\mathcal{D}$.  For $a\in A$ and $d\in D$, we have that $ad\in C_c$ by the above paragraph, and $a+P(d)\subset C_c$ by property (3) of the Claim, completing the proof.
\end{proof}

Finally, Theorem \ref{thm:main} follows by combining Propositions \ref{geo proof}, \ref{linear proof}, and \ref{prop iterated full}.

\section{Open problems}

In addition to Question \ref{double vdw} and the problem of Kra, Moreira, Richter, and Robertson mentioned in the intro, there are several other natural ways to potentially extend this work.  For example, in a recent paper with Sabok \cite{bowen.sabok} we showed that any finite coloring of $\mathbb{Q}$ contains a monochromatic set of the form $\{a,b,ab,a+b\}.$  Can we prove a common generalization of this fact and Corollary \ref{cor iterated}?  The $3$ term case of this would be the following (recall that, given a set of numbers $A,$ by $FS(A)$ we mean the set of all finite non-repeating sums of elements of $A$ and by $FP(A)$ we mean the set of all non-repeating finite products):

\begin{question}
Does every finite coloring of $\mathbb{Q}$ contain a monochromatic set of the form $$FS(a,b,c)\cup FP(a,b,c)\cup (a+ FP(b,c))\cup (a\cdot FS(b,c))?$$
\end{question}



In another direction, Corollary \ref{cor iterated} shows that in any finite coloring of the naturals we can find $a_1,...,a_n$ such that no matter how we place $+$ and $\cdot$ between the terms, the resulting expression is the same color \textit{so long as the expression is calculated with respect to the increasing bracketing}.  

Is a similar Ramsey theorem still true if we allow for any bracketing instead of the increasing one?  The first open case of this is the following:

\begin{question}
Suppose the naturals are finitely colored.  Are there $a,b,c,d\in\mathbb{N}$ such that the set $$\{(a\circ_1 b)\circ_2(c\circ_3 d) : \circ_i\in \{+,\cdot\} \}$$

is monochromatic?
\end{question}






\section*{Acknowledgments} 
Thanks to Zach Hunter, Marcin Sabok, and the anonymous referee for many helpful comments and corrections on earlier versions of this paper.

\bibliographystyle{amsplain}




\begin{dajauthors}
\begin{authorinfo}[bowen]
  Matt Bowen\\
  Oxford University\\
  Oxford, United Kingdom\\
  matthew\imagedot{}bowen\imageat{}maths\imagedot{}ox\imagedot{}ac \\
  \url{https://sites.google.com/view/matt-bowen/home}
\end{authorinfo}
\end{dajauthors}

\end{document}